\documentclass[11pt]{amsart}  

\usepackage[utf8]{inputenc}
\usepackage{amsmath}
\usepackage{amsfonts}
\usepackage{amssymb}
\usepackage{latexsym}
\usepackage{graphicx}
\usepackage{subfigure}
\usepackage{color}
\usepackage{mathtools}
\usepackage[backref=page]{hyperref}
\usepackage{verbatim}
\usepackage[all]{xy}
\usepackage{pdfsync}
\usepackage{youngtab}
\usepackage{ytableau}
\usepackage{young}
\usepackage{tikz}
\usepackage{cancel} 
\usepackage[normalem]{ulem} 
\usepackage{graphicx}
\usepackage{forest}
\usepackage{array}
\usepackage{extarrows}

\theoremstyle{plain}
\newtheorem{theorem}{Theorem}[section]

\newtheorem{lemma}[theorem]{Lemma}
\newtheorem{corollary}[theorem]{Corollary}

\newtheorem{example}[theorem]{Example}

\newtheorem{remark}[theorem]{Remark}

\numberwithin{equation}{section}
\newcommand{\bC}{\mathbb{C}}
\newcommand{\bN}{\mathbb{N}}

\newcommand{\bR}{\mathbb{R}}
\newcommand{\bZ}{\mathbb{Z}}

\newcommand{\tcr}[1]{\textcolor{red}{#1}}

\newcommand{\suchthat}{\;|\;}

\newcommand{\p}[1]{\mathcal #1} 

\newcommand{\weight}{{\mathsf {wt}}} 

\definecolor{emphcol}{rgb}{0.0, 0.5, 2.0}

\DeclareMathOperator{\GCD}{GCD}
\DeclareMathOperator{\PF}{PF}

\DeclareMathOperator{\md}{mod}

\newcommand{\shift}{\mathsf{shift}} 
\newcommand{\rot}{\mathsf{rotate}} 
\newcommand{\sort}{\mathsf{sort}} 

\newcommand{\pf}{\mathrm{Park}} 

\begin{document}

\title[The trimmed standard permutahedron extends the parking space]{Trimming the permutahedron to extend the parking space}
\author{Matja\v z Konvalinka}
\address{Department of Mathematics, University of Ljubljana \& Institute of Mathematics, Physics and Mechanics, Ljubljana, Slovenia}
\email{\href{mailto:matjaz.konvalinka@fmf.uni-lj.si}{matjaz.konvalinka@fmf.uni-lj.si}}
\author{Robin Sulzgruber}
\address{Department of Mathematics and Statistics, York University, Toronto, Canada}
\email{\href{mailto:rsulzg@yorku.ca}{rsulzg@yorku.ca}}
\author{Vasu Tewari}
\address{Department of Mathematics, University of Pennsylvania, Philadelphia, PA 19104, USA}
\thanks{The first author acknowledges the financial support from the Slovenian Research Agency (research core funding No. P1-0294).}
\email{\href{mailto:vvtewari@math.upenn.edu}{vvtewari@math.upenn.edu}}

\begin{abstract}
  Berget and Rhoades asked whether the permutation representation obtained by the action of $S_{n-1}$ on parking functions of length $n-1$ can be extended to a permutation action of $S_{n}$.
  We answer this question in the affirmative. We realize our module in two different ways. The first description involves binary Lyndon words and the second involves the action of the symmetric group on the lattice points of the trimmed standard permutahedron.
  \end{abstract}

\maketitle
\section{Introduction}

In their study of an extension of the classical parking function representation, Berget and Rhoades asked \cite[Section 4]{BR14} whether the permutation action of $S_{n-1}$ on parking functions of length $n-1$ could be extended to a permutation action of $S_{n}$.
The extension  $V_{n-1}$ in \cite{BR14} is realized by considering the $\bC$-span of a distinguished set of polynomials in $n$ variables first studied by Postnikov and Shapiro \cite{Pos04}. While the aforementioned set of polynomials is $S_n$-stable, it does not form a basis for $V_{n-1}$ in general.
Berget and Rhoades work with a basis for $V_{n-1}$ \emph{that is not $S_n$-stable.}
To establish that the restriction of $V_{n-1}$ from $S_{n}$ to $S_{n-1}$ is indeed Haiman's parking function representation \cite{Hai94} (henceforth referred to as $\pf_{n-1}$), Berget and Rhoades use Gr\"obner-theoretic techniques to construct a linear subspace of $V_{n-1}$ with a $S_{n-1}$-stable monomial basis indexed by parking functions.s

\medskip

Our point of departure is a particular permutahedron in $\bR^n$ whose set of lattice points is equinumerous with the set of parking functions on length $n-1$, thereby providing a plausible candidate.
Given a tuple $\lambda=(\lambda_1\geq \cdots \geq \lambda_n)$, we define the \emph{permutahedron} $P_{\lambda}\subset \bR^n$ to be the convex hull of the $S_n$ orbit of $\lambda$.
We denote the set of lattice points $P_{\lambda}\cap \bZ^n$ by $\mathrm{Lat}(P_{\lambda})$.
For $n\geq 2$, define $\delta_n$ to be the partition $(n-2,\dots,1,0,0)$.
It is clear that $S_n$ acts on $\mathrm{Lat}(P_{\delta_n})$.
Let $\gamma_n$ denote the associated representation.
Here is our main result which answers the question posed by Berget and Rhoades.
\begin{theorem}\label{thm:main}
	We have that
	\[
	\mathrm{Res}^{S_n}_{S_{n-1}}(\gamma_n)=\pf_{n-1}.
	\]
\end{theorem}
Thus $\gamma_n$ is a permutation representation that extends the parking function representation.
Furthermore,  a conjecture of the first and third author \cite[Conjecture 3.1]{KT20} may be restated as claiming that $\gamma_n$ is isomorphic to the ungraded Berget-Rhoades representation $V_{n-1}$.

Our approach is indirect and builds off of earlier work \cite{KT20} by the first and third author wherein a family of $S_n$-representations $\widehat{\PF}_{n,c}$ that restrict to $\pf_{n-1}$ is constructed.
For an appropriately chosen value of $c$, this representation is isomorphic to what we consider here.
We give an explicit $h$-positive expansion for the Frobenius characteristic of this representation in terms of binary Lyndon words satisfying a straightforward constraint.
Our proof goes via an intermediate module $\p C_{m,n}$ that we analyze in depth as well.
The representation $\gamma_n$  is obtained by identifying elements of {$\p C_{1,n}$}
up to a natural equivalence relation.
Finally, we can compute the character of $\gamma_n$ by appealing to \cite[Theorem 3.2]{KT20}.
This does not appear to be a straightforward task from the definition of $\gamma_n$.
The character values allow us to make a connection with recent work of Ardila, Schindler and Vindas-Mel\'endez \cite{ASV18}, which we state next.
\begin{corollary}\label{cor:connection_with_equivariant_volumes}
Let $\Pi_n$ denote the standard permutahedron $P_{(n-1,\dots,1,0)}$.
Given $\sigma\in S_n$ with cycle type $(\lambda_1,\dots,\lambda_{\ell})$, let $\Pi_n^{\sigma}$ denote the set of points in $\Pi_n$ that are fixed by $\sigma$. Suppose $\GCD(\lambda_1,\dots,\lambda_{\ell})=1$. Then the normalized volume of $\Pi_n^{\sigma}$ is equal to the number of lattice points in $P_{\delta_n}$ fixed by $\sigma$.
\end{corollary}
If $\sigma$ is the identity permutation, then Corollary~\ref{cor:connection_with_equivariant_volumes} says that normalized volume of the standard permutahedron in $\bR^n$ is equal to the number of lattice points in $P_{\delta_n}$. The former is well known to equal $n^{n-2}$ \cite{St91}.
Thus, in this specific instance, our result reduces to a special case of \cite[Corollary 11.5]{Pos09}.

\medskip

For maximum generality, we work in the setting of rational parking functions for the majority of this paper. In Section~\ref{sec:permutahedron}, we specialize to arrive at Theorem~\ref{thm:main}.

\section{The setup}\label{sec:setup}
To keep our exposition brief, we refer the reader \cite[Chapter 7]{St99} for all  notions pertaining to the combinatorics of symmetric functions which are not defined explicitly here.

\subsection{The modules \texorpdfstring{$\p C_{m,n}$}{C m,n} and \texorpdfstring{$\widehat{\p C}_{m,n}$}{CHat m,n}}
Given positive integers $m$ and $n$, set $N\coloneqq mn$ and
\begin{align*}
c_{m,n}\coloneqq \frac{(N-2)(n-1)}{2}.
\end{align*}
This given, consider the set of $N$-tuples defined as follows:
\begin{align*}
\p C_{m,n}\coloneqq \{(x_1,\dots,x_{N})\suchthat 0\leq x_i\leq n-1, \sum_{1\leq i\leq N}x_i=c_{m,n} \: (\md n)\}.
\end{align*}
Clearly, $|\p C_{m,n}|=n^{N-1}$.
Geometrically, one may interpret $C_{m,n}$ to be the set of lattice points in the cube $[0,n-1]^N$ in $\bR^n$ that lie on
{certain} translates of the hyperplane $x_1+\dots+x_N=0$.
Note that $S_N$ acts on $\p C_{m,n}$ by permuting coordinates and we denote the resulting permutation action by $\tau_{m,n}$.
We abuse notation and use $\p C_{m,n}$ to denote both the set and the resulting $S_N$-module.

\medskip

Let $\Lambda_k$ denote the set of tuples $\lambda=(\lambda_1\geq \cdots \geq \lambda_k)$ in $\bN^k$.
Here $\bN$ denotes the set of  nonnegative integers.
We refer to elements of $\Lambda_k$ as \emph{partitions}. Given $(\lambda_1,\dots,\lambda_k)\in \Lambda_k$, we refer to $\lambda_i$'s as the \emph{parts} of $\lambda$. In particular, we consider $0$ to be a part.
Given any sequence $\mathbf{x}=(x_1,\dots,x_k)\in \bN^k$, we define $\sort(\mathbf{x})$ to be the partition obtained by sorting $\mathbf{x}$ in nonincreasing order.
Let
\[
Y_{m,n}\coloneqq \Lambda_N\cap \p C_{m,n}.
\]
Clearly, elements of $Y_{m,n}$ index the orbits of $\p C_{m,n}$ under $\tau_{m,n}$.

\medskip

By drawing $\lambda \in Y_{m,n}$ as a Young diagram in French notation so that the lower left corner coincides with the origin in $\bZ^2$, we  may identify $\lambda$ with a lattice path $L_{\lambda}$  that starts at $(n,0)$, ends at $(0,N)$, and takes \emph{vertical} and \emph{horizontal} steps of unit length.
All coordinates here are Cartesian.
It will be convenient to extend $L_{\lambda}$ to an infinite path $L_{\lambda}^{\infty}$ by repeating $L_{\lambda}$.
Figure~\ref{fig:L_lambda} depicts $\lambda=(2,2,1,1,0)\in Y_{1,5}$. The shaded region represents the $5\times 5$ box where $\lambda$ is drawn, the red path depicts a fragment of $L_{\lambda}^{\infty}$, and the thickened subpath represents $L_{\lambda}$.

\begin{figure}[!ht]
\begin{tikzpicture}[scale=.25]
    \draw[gray,very thin] (-5,-1) grid (6,10);
    \draw[line width=0.15mm, black, <->] (-5,1)--(6,1);
    \draw[line width=0.15mm, black, <->] (0,-1)--(0,10);
    \draw[fill=gray,opacity=0.35] (0, 1) rectangle (5,6) {};
    \draw[fill=blue,opacity=0.45] (0, 1) rectangle (1,5) {};
    \draw[fill=blue,opacity=0.45] (1, 1) rectangle (2,3) {};


    \draw[red, line width=0.4mm] (5,1)--(2,1)--(2,3)--(1,3)--(1,5)--(0,5)--(0,6);
    \draw[red,line width=0.25mm](0,6)--(-3,6)--(-3,8)--(-4,8)--(-4,10)--(-5,10);
    \draw[red, line width=0.25mm] (5,1)--(5,0)--(6,0)--(6,-1);

  \end{tikzpicture}
  \caption{Path corresponding to $\lambda=(2,2,1,1,0)$.}
  \label{fig:L_lambda}
\end{figure}
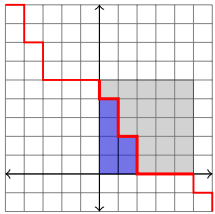


We define the \emph{shift} map $\shift$ mapping $\{0,\dots,n-1\}^N$ to itself via the rule
\begin{align*}
\shift(x_1,\dots,x_N)\coloneqq (x_1+1,\dots,x_N+1),
\end{align*}
where addition is performed modulo $n$.
It is clear that $\shift^n$ is the identity map.
This map allows us to define an equivalence relation $\sim$ on $\p C_{m,n}$ by declaring two sequences to be \emph{equivalent} if one is obtained by applying $\shift^j$ to the other  for some $j\in \bN$.
Since the sum of the coordinates remains invariant modulo $n$ upon applying the shift map, our $S_N$ action on $\p C_{m,n}$ descends to an action $\widehat{\tau}_{m,n}$ on the set of equivalences classes $\p C_{m,n}/\sim$.
We denote this set (and the associated $S_N$ module) by $\widehat{\p C}_{m,n}$.
Since every equivalence class has $n$ elements, we have that $|\widehat{\p C}_{m,n}|=n^{N-2}$.

\begin{remark}\label{rem:restriction}
\emph{
The careful reader should note that our construction works equally well with $c_{m,n}$ replaced by any integer. 
 Thus, one obtains a family of $S_N$ modules in this manner.
The analog of the set $\widehat{\p C}_{m,n}$ has the property that every equivalence class therein  has a unique element $(x_1,\dots,x_N)$ so that $(x_1,\dots,x_{N-1})$ is a rational parking function. See discussion in \cite[Section 5 ]{KT20} to this end.
It follows that the modules under consideration in this article are a special case of those studied in loc.\ cit., the choice $c_{m,n}$ having been made to answer the question of Berget and Rhoades.
This choice  is not merely a fortuitous coincidence as $c_{m,n}$ is closely related to the area statistic on parking functions, and the latter already plays a role in \cite{BR14}. 
The one pertinent upshot of this discussion is that $\widehat{\p C}_{1,n}$ restricts to $\pf_{n-1}$.
}
\end{remark}

\subsection{Binary words and \texorpdfstring{$Y_{m,n}$}{Y m,n}}
We now interpret the partitions in $Y_{m,n}$ as certain words in the alphabet $\{0,1\}$ as this will shed more light into their structure.
By reading $L_{\lambda}$ from right to left and recording a $0$ (respectively $1$) for each horizontal (respectively vertical) step, we obtain a word $w_{\lambda}$ of length $(m+1)n$ in  $\{0,1\}$.
Clearly, $w_{\lambda}$ begins with a $0$ and has $mn$ 1s and $n$ 0s.
We refer to any word in the alphabet $\{0,1\}$ with the property that the number of $1$s is $m$ times the number of $0$s as \emph{$m$-balanced}.
We denote the length of a word $w$ by  $|w|$. For the partition $\lambda=(2,2,1,1,0)\in Y_{1,5}$ depicted in Figure~\ref{fig:L_lambda}, we have that $w_{\lambda}=0001101101$.

\medskip

We now associate a \emph{weight} $\weight(w)$ with any $m$-balanced word $w=w_1\dots w_{(m+1)n}$:
\begin{align*}
\weight(w)\coloneqq \sum_{w_i=1}i.
\end{align*}
This given, define
\begin{align*}
B_{m,n}\coloneqq\{w\in \{0,1\}^{(m+1)n}\suchthat w \text{ $m$-balanced}, w_1=0, \weight(w)=-1 \: (\md n) \}.
\end{align*}
One may check that $w_{(2,2,1,1,0)}=0001101101$ belongs to  $B_{1,5}$ as its weight equals $34$, which indeed is $-1 \: (\md 5)$.

\medskip

Before establishing a couple of lemmas that emphasize the importance of $B_{m,n}$, we need some notions from the combinatorics on words \cite{Lo97}.
Given a word $w=w_1\cdots w_k$, define $\rot(w)\coloneqq w_2\cdots w_kw_1$.
Clearly, $\rot^{k}(w) = w$.
We say that $w$ is \emph{primitive} (or \emph{aperiodic}) if no proper cyclic rotation of $w$ coincides with $w$.
In other words, $w=w_1\dots w_{k}$ is primitive if $\rot^j(w)\neq w$ for $1\leq j\leq k-1$.
For instance, $0101$ is not primitive while $0011$ is.
If $w$ is not primitive, it may be written as $w=\widehat{w}^{{k}/{d}}$ for some primitive $\widehat{w}$ and $d$ a proper divisor of $k$ (that is, $d$ cannot equal $k$).
We say that two words are \emph{conjugate} if one is obtained as a cyclic rotation of the other. Observe that a conjugate of a primitive word is again primitive.

\begin{lemma}\label{lem:w_lambda_primitive}
	Every $w\in B_{m,n}$ is primitive.
\end{lemma}
\begin{proof}
Towards establishing a contradiction, suppose $w$ is not primitive.
Then $w=\widehat{w}^{(m+1)n/d}$ where $\widehat{w}$ is primitive and $d$ is a proper divisor of $(m+1)n$.
We claim that $\widehat{w}$ is $m$-balanced as well.
Indeed, say $\widehat{w}$ possesses $r$ 0s.
It must be that
\begin{align}
\frac{(m+1)n}{d}r=n,
\end{align}
which implies that  $r=\frac{d}{m+1}$. Thus, the number of $1$s in $\widehat{w}$ is $\frac{md}{m+1}$, implying that $\widehat{w}$ is $m$-balanced.
Note in particular that  $(m+1)|d$.

\medskip

Suppose $\weight(\widehat{w})$ equals $M$.
Then we have that
\begin{align}
\weight(w)&=M+\left(M+\frac{md}{m+1}d\right)+\left(M+\frac{md}{m+1}2d\right)+\cdots+\left(M+\frac{md}{m+1}\left(\frac{(m+1)n}{d}-1\right)d\right)\nonumber\\
&=\frac{(m+1)n}{d}M+\frac{md^2}{m+1}\binom{\frac{(m+1)n}{d}}{2},
\end{align}
which, modulo $n$, translates to the  equality
\begin{align}\label{eqn:weight_alternative_1}
\weight(w)=\frac{(m+1)n}{d}M-\frac{mnd}{2}.
\end{align}
Since $(m+1)|d$, we know that $md$ is even.
Thus \eqref{eqn:weight_alternative_1} simplifies to
\begin{align}
\label{eqn:weight_alternative_2}
\weight(w)=\frac{(m+1)n}{d}M \: (\md n).
\end{align}

\medskip

Since $w\in B_{m,n}$, we know that $\weight(w)=-1 \: (\md n)$. This in conjunction with
\eqref{eqn:weight_alternative_2} implies that $M$ satisfies
\begin{align}\label{eqn:no_solution}
	\frac{(m+1)n}{d}M=-1 \: (\md n).
\end{align}
Writing $(m+1)n/d$ as $\frac{n}{d/(m+1)}$, and recalling that $d$ is a proper divisor of $(m+1)n$, we conclude that $\GCD(\frac{(m+1)n}{d},n)\geq 2$.
In particular, \eqref{eqn:no_solution} has no solutions, and we have established that $w$ is primitive.
\end{proof}

We are now ready to relate $B_{m,n}$ to $Y_{m,n}$.

\begin{lemma}
$B_{m,n}$ and $Y_{m,n}$ have the same cardinality.
\end{lemma}
\begin{proof}
	We claim that the correspondence $\lambda\mapsto w_{\lambda}$ is a bijection from $Y_{m,n}$ to $B_{m,n}$.
	To this end, we first show that $w_{\lambda}\in B_{m,n}$. For convenience, set $w=w_1\dots w_{(m+1)n}\coloneqq w_{\lambda}$.
	It is immediate that $w_1=0$ as $\lambda_1\leq n-1$, and that $w$ is $m$-balanced.
	Thus we need to check that $\weight(w)=-1 \: (\md n)$.
	It is easy to see that $w_j=1$ if and only if $j=n-\lambda_{i}+i$ for some (unique) $i$.
	Thus we obtain
\begin{align}
\weight(w)=\sum_{1\leq i\leq N}(n-\lambda_{i}+i)=Nn+\binom{N+1}{2}-|\lambda|.
\end{align}
Since $N = m n$ and $|\lambda| = (N-2)(n-1)/2$, we have 
\begin{align}
\weight(w)=Nn + \binom m 2 n^2 + mn+n-1.
\end{align}
Thus we conclude that $\mathsf{wt}(w)=-1 \: (\md n)$, and therefore $w\in B_{m,n}$.

\medskip

It is clear that this correspondence is an injection from $Y_{m,n}$ to $B_{m,n}$.
That this is a bijection follows because this correspondence is easily reversible,  and one may obtain a partition for every word in $B_{m,n}$.
That this partition belongs to $Y_{m,n}$ follows by reading the earlier string of equalities backwards.
\end{proof}

We use this correspondence to obtain a `closed form' for $|Y_{m,n}|$.
\begin{corollary}
The number of orbits of $\p C_{m,n}$ under $\tau_{m,n}$ equals the number of $N$-elements subsets of $[(m+1)n-1]\coloneqq \{1,\dots,(m+1)n-1\}$ whose subset sum equals $-1 \: (\md n)$. More explicitly, we have
\begin{align*}
|Y_{m,n}|=\frac{1}{n}\sum_{d|n} (-1)^{m(n+d)}\mu(n/d)\binom{(m+1)d-1}{md}.
\end{align*}
\end{corollary}
\begin{proof}
Given $w=w_1\dots w_{(m+1)n}\in B_{m,n}$, associate an $N$-element subset $S_w$ of $[(m+1)n-1]$ by
\begin{align*}
	S_w=\{j-1\suchthat w_j=1\}.
\end{align*}
It is clear that the sum of elements in $S_w$ is $-1 \: (\md n)$, and that this correspondence sets up a bijection between $B_{m,n}$ and $N$-elements subsets of $[(m+1)n-1]$ with subset sum equal to $-1 \: (\md n)$.
It remains to count such subsets, and we appeal to \cite[Theorem 1.1]{Ch20} to this end.
Setting $u=m+1$, $v=m$ and $c=n-1$ in loc.\ cit.\ implies that
\begin{align}
B_{m,n}&=\frac{1}{(m+1)n}\sum_{d|n} (-1)^{m(n+d)}\mu(n/d)\binom{(m+1)d}{md}\nonumber\\
&=\frac{1}{n}\sum_{d|n} (-1)^{m(n+d)}\mu(n/d)\binom{(m+1)d-1}{md},
\end{align}
which completes the proof.
\end{proof}

\section{The \texorpdfstring{$h$}{h}-positivity of \texorpdfstring{$\mathrm{Frob}(\widehat{\tau}_{m,n})$}{Frob(tauHat m,n)}}
In this section we establish the $h$-positivity of the action $\widehat{\tau}_{m,n}$ on $\widehat{\p C}_{m,n}$.
We do this by determining representatives of equivalence classes in $\widehat{C}_{m,n}$ that carry the natural action of $S_N$ which permutes coordinates.
As remarked earlier,  in \cite{KT20}, we made a choice of representatives with the property that the first $N-1$ coordinates give a rational parking function.
Unfortunately, this choice does not give rise to an $S_N$-stable set typically.

\begin{lemma}\label{lem:orbit_by_shifts}
	For every $\lambda\in Y_{m,n}$, the cardinality of the set $\{\sort\circ \shift^j(\lambda)\suchthat 0\leq j\leq n-1\}$ is $n$.
\end{lemma}
\begin{proof}
	We exploit a nice connection between rotations of $m$-balanced words and the shift map applied to elements of $Y_{m,n}$.
	Consider the lattice path $L_{\lambda}$ associated to $\lambda\in Y_{m,n}$, as well as its infinite extension $L_{\lambda}^{\infty}$.
	Label the $(m+1)n$ steps of $L_{\lambda}$ with integers $1$ through $(m+1)n$ going right to left.

	\medskip

	Let the labels of the horizontal steps be $a_0<a_1<\cdots<a_{n-1}$.
	This given, here is how one may compute $\sort\circ \shift^j(\lambda)$ for $0\leq j\leq n-1$.
Consider the subpath of $L_{\lambda}^{\infty}$ of length $(m+1)n$ that begins with the horizontal step labeled $a_j$ and proceeds northwest.
This subpath may be treated as $L_{\widetilde{\lambda}}$ (after a potential translation) for a unique partition $\widetilde{\lambda}\in Y_{m,n}$.
It is not hard to see that $\widetilde{\lambda}=\sort\circ \shift^j(\lambda)$ by realizing that $\sort\circ\shift$ is essentially a `rotation' given that $\lambda$ is weakly decreasing from left to right. For instance, consider Figure~\ref{fig:L_lambda_prime} where the partition $\sort\circ\shift^3(\lambda)$ is depicted by the orange shaded region in the translated $5\times 5$ box for $\lambda=(2,2,1,1,0)$.

\medskip

From the fact that $L_{\lambda}^{\infty}$ is built by periodically repeating $L_{\lambda}$, it follows that
\begin{align}
w_{\widetilde{\lambda}}=\rot^{a_j-1}(w_{\lambda}).
\end{align}
Now, recall that $w_{\lambda}$ is primitive by Lemma~\ref{lem:w_lambda_primitive}, and thus $w_{\widetilde{\lambda}}$ and $w_{\lambda}$ are distinct.
This in turn implies that $\widetilde{\lambda}$  and $\lambda$ are distinct as well.
It follows that the set $\{\sort\circ \shift^j(\lambda)\suchthat 0\leq j\leq n-1\}$ indeed has $n$ elements, one for each conjugate of $w_{\lambda}$ beginning with a $0$.
\end{proof}
\begin{figure}[!ht]
\begin{tikzpicture}[scale=.25]
    \draw[gray,very thin] (-5,-1) grid (6,10);
    \draw[line width=0.15mm, black, <->] (-5,1)--(6,1);
    \draw[line width=0.15mm, black, <->] (0,-1)--(0,10);
    \draw[fill=gray,opacity=0.15] (0, 1) rectangle (5,6) {};
    \draw[fill=gray,opacity=0.55] (-3, 3) rectangle (2,8) {};
    \draw[fill=blue,opacity=0.20] (0, 1) rectangle (1,5) {};
    \draw[fill=blue,opacity=0.20] (1, 1) rectangle (2,3) {};
    \draw[fill=orange,opacity=0.35] (-3, 3) rectangle (0,6) {};
    \draw[fill=orange,opacity=0.35] (0, 3) rectangle (1,5) {};



    \draw[red, line width=0.4mm] (5,1)--(2,1)--(2,3)--(1,3)--(1,5)--(0,5)--(0,6);
    \draw[red,line width=0.25mm](0,6)--(-3,6)--(-3,8)--(-4,8)--(-4,10)--(-5,10);
    \draw[red, line width=0.25mm] (5,1)--(5,0)--(6,0)--(6,-1);

  \end{tikzpicture}
  \caption{The partition $\sort\circ\shift^3(\lambda)$ realized by $L_{\widetilde{\lambda}}$ where $\lambda=(2,2,1,1,0)$.}
  \label{fig:L_lambda_prime}
\end{figure}
We now exploit this lemma to extract a set of representatives for the equivalence classes in $\widehat{\p C}_{m,n}$ that is $S_N$-stable.

Given $w\in B_{m,n}$, we denote the associated partition $\lambda$ satisfying $w_{\lambda}=w$ as $\lambda_w$.
Recall that the conjugacy class of $w$ consists of all $(m+1)n$ cyclic rotations thereof.
Amongst these cyclic rotations, there is a unique lexicographically smallest word, where the order is inherited by declaring $0<1$.
Such a word is known as a \emph{Lyndon word}.
Observe the crucial fact that Lyndon word in the conjugacy class of $w\in B_{m,n}$ must itself belong to $B_{m,n}$, as it must begin with a $0$ and cyclic rotations preserve weights.
We denote the set of Lyndon words in $B_{m,n}$ by $B_{m,n}^{L}$.

\begin{example}\emph{
Consider $\p C_{1,4}$. The 8 partitions in $Y_{1,4}$ (commas and parentheses suppressed) are given below. Those in the same column are obtained by applying $\sort\circ \shift^j$ for $0\leq j\leq 3$ to the highlighted partition.
\begin{align*}
\begin{array} {cc}
\tcr{2100} & \tcr{1110}\\
3211 & 2221\\
3220 & 3332\\
3310 & 3000
\end{array}
\end{align*}
It can be checked that the words $w_{\lambda}$ corresponding to the highlighted partitions are indeed Lyndon, and thus
\begin{align*}
	B_{1,4}^L=\{00101011,00011101\}.
\end{align*}
More importantly, since elements of $Y_{m,n}$ index orbits of $\p C_{m,n}$ and $\widehat{\p C}_{m,n}$ is obtained by identifying elements of $\p C_{m,n}$ up to shifts, it follows that the orbits of $2100$ and $1110$ generate a system of representatives for equivalence classes in $\widehat{\p C}_{m,n}$.
This is the underlying idea of what follows.
}
\end{example}

Let $\p O_{\lambda_w}$ denote the $S_N$-orbit of $\lambda_w$ for $w\in B_{m,n}$.
We claim that the set of elements of $\p C_{m,n}$ that belong to the orbit of $\lambda_w$ for a Lyndon word $w\in B_{m,n}$ gives a complete set of  representatives for equivalence classes in $\widehat{\p C}_{m,n}$.
Indeed, we know that
\begin{align}\label{eqn:disjoint_over_all}
\p C_{m,n}&=\coprod_{w\in B_{m,n}} \p O_{\lambda_w},
\end{align}
and by invoking Lemma~\ref{lem:orbit_by_shifts} to rewrite the right-hand side of \eqref{eqn:disjoint_over_all} we get that
\begin{align}\label{eqn:disjoint_over_lyndon}
\p C_{m,n}&=\coprod_{w\in B_{m,n}^L}\coprod_{0\leq j\leq n-1}\p O_{\sort\circ\shift^j(\lambda_w)}\nonumber\\
&=\coprod_{w\in B_{m,n}^L}\coprod_{0\leq j\leq n-1}\p O_{\shift^j(\lambda_w)}.
\end{align}
Since $\shift$ commutes with the action of $S_N$, we can rewrite \eqref{eqn:disjoint_over_lyndon} as
\begin{align}\label{eqn:disjoint_over_lyndon_2}
\p C_{m,n}&=\coprod_{w\in B_{m,n}^L}\coprod_{0\leq j\leq n-1} \shift^{j}(\p O_{\lambda_w}),
\end{align}
where we interpret $\shift^{j}(\p O_{\lambda_w})$ as the set obtained by applying $\shift^j$ to all elements in $\p O_{\lambda_w}$.
Since $\widehat{\p C}_{m,n}$ is obtained by identifying sequences in $\p C_{m,n}$ up to shifts, \eqref{eqn:disjoint_over_lyndon_2} tells us that we may identify $\widehat{\p C}_{m,n}$ with $\coprod_{w\in B_{m,n}^L}\p O_{\lambda_w}$, and thus $\widehat{\tau}_{m,n}$ is indeed the permutation action on the latter set.
We are now ready to record an immediate consequence of this argument.

\begin{theorem}\label{thm:h_positivity}
	To each $w\in B_{m,n}^L$, associate a sequence $\mathrm{c}(w)$ that records the lengths of the maximal runs of $1$s in w. Then we have that
	\[
	\mathrm{Frob}(\widehat{\tau}_{m,n})=\sum_{w\in B_{m,n}^L}h_{\sort(\mathrm{c}(w))}.
	\]
	In particular, we have that the number of orbits of $\widehat{\p C}_{m,n}$ under $\widehat{\tau}_{m,n}$ is given by
	\[
	|B_{m,n}^{L}|=\frac{|Y_{m,n}|}{n}=\frac{1}{n^2}\sum_{d|n} (-1)^{m(n+d)}\mu(n/d)\binom{(m+1)d-1}{md}.
	\]
\end{theorem}
Recall that in \cite[Theorem 6.1]{KT20}, the number of orbits was computed by way of explicit character values.
See also \cite[Section 5]{Ray18} for a topological interpretation for the numbers $|B_{m,n}^{L}|$.

\begin{example}
\emph{
Consider $m=2$ and $n=3$.
The three Lyndon words in $B_{m,n}^{L}$ are
\[
001011111 \hspace{5mm} 001111011 \hspace{5mm} 010110111.
\]
The corresponding $\sort(\mathrm{c}(w))$ are $51$, $42$, and $321$. Theorem~\ref{thm:h_positivity} implies that
\[
\mathrm{Frob}(\widehat{\tau}_{m,n})=h_{51}+h_{42}+h_{321}.
\]}
\end{example}

%
%

\section{The trimmed standard permutahedron}\label{sec:permutahedron}
In this section we focus on the case $m=1$, or equivalently, $N=n$. Hence we suppress the $m$ from all notions introduced earlier.
Our goal is to establish Theorem~\ref{thm:main} stated in the introduction.

\medskip

Recall that given $\lambda\coloneqq (\lambda_1\geq \cdots\geq  \lambda_n)\in \bZ^n$, we let $P_{\lambda}$ denote the polytope in $\bR^n$ defined by considering the convex hull of the $S_n$ orbit of $\lambda$.
The $P_{\lambda}$'s are referred to as \emph{usual permutahedra}.
The set of lattice points in $P_{\lambda}$, that is $\mathrm{Lat}(P_{\lambda})$, is clearly $S_n$-stable, and one obtains a natural class of $S_n$-modules in this manner. Furthermore, since the stabilizer of any point in $\mathrm{Lat}(P_{\lambda})$ is a Young subgroup of $S_n$, we are guaranteed $h$-positivity of the associated Frobenius characteristics.
It is a priori unclear whether these modules are of any value other than the intrinsic one. In what follows, we discuss the case of a special permutahedron, and show that its set of lattice points indexes the orbits of $\widehat{\p C}_{n}$ under the action of $\widehat{\tau}_n$.

\medskip


Fix $n\geq 2$.
Let $\delta_n=(n-2,\dots,1,0,0)\in \bN^n$.
We are interested in the $S_n$ action on $\mathrm{Lat}(P_{\delta_n})$, which we denote by $\gamma_n$.
The reason behind this is the  equality
\begin{align}\label{eqn:lattice_point_count}
|\mathrm{Lat}(P_{\delta_n})|=n^{n-2}.
\end{align}
We briefly explain how to arrive at this equality through work of Postnikov \cite{Pos09}, leaving it to the reader to check loc.\ cit.\ for further details.
Consider the \emph{standard permutahedron} $P_{(n-1,\dots,1,0)}$.
The Minkowski difference of the standard permutahedron with the standard simplex $P_{(1,0,\dots,0)}$ is another permutahedron $P_{(n-2,n-2,\dots,1,0)}$. Following \cite[Definition 11.2]{Pos09}, we refer to the latter as the \emph{trimmed} standard permutahedron. By \cite[Corollary 11.5]{Pos09}, we have that
\begin{align}
	|\mathrm{Lat}(P_{(n-2,n-2,\dots,1,0)})|=n^{n-2}.
\end{align}
That this is equivalent to the equality in \eqref{eqn:lattice_point_count} is because translating $P_{(n-2,n-2,\dots,1,0)}$ by $(n-2,\ldots,n-2)$ followed by negating all coordinates maps it to $P_{\delta_n}$.
It is also clear that this map is $S_n$-equivariant, so the $S_n$ action on $P_{\delta_n}$ is isomorphic to that on $P_{(n-2,n-2,\dots,1,0)}$, which also explains the title of this section.

\medskip

The right-hand side of \eqref{eqn:lattice_point_count} naturally raises the question whether this $S_n$-action is related to the parking function representation. Indeed, as we shall soon establish, upon restricting this action to $S_{n-1}$ we recover the parking function representation.

\begin{example}
\emph{
Suppose $n=4$. Then $\delta_n=(2,1,0,0)$. The 16 elements in $\mathrm{Lat}(P_{\delta_n})$ are given by the orbits of $(2,1,0,0)$ and $(1,1,1,0)$. It follows that
\[
\mathrm{Frob}(\gamma_4)=h_{211}+h_{31},
\]
which upon restricting to $S_3$ gives $h_3+3h_{21}+h_{111}$. This last expression is the Frobenius characteristic of the $S_3$ action on parking functions of length $3$.
An alternative perspective is by projecting $P_{\delta_4}$ onto the hyperplane $x_4=0$ in $\bR^4$ and realizing that $S_3$ acts on the lattice points of the resulting polytope. The $h_{111}$ term comes from the orbit of the point $(2,1,0)$, the three $h_{21}$ terms come from the orbits of $(1,0,0)$, $(2,0,0)$, and $(1,1,0)$ respectively, and the $h_3$ term comes from the orbit of $(1,1,1)$.
}
\end{example}

To establish that $\gamma_n$ is isomorphic to $\tau_{n}$, we identify representatives of the equivalence classes in $\widehat{\p C}_{n}$ that belong to $\mathrm{Lat}(P_{\delta_n})$.
Clearly, the $S_n$ action on $\mathrm{Lat}(P_{\delta_n})$ has orbits indexed by  elements of
 $\mathrm{Par}_{\leq \delta_n}$, which we defined to be the set of lattice points $(\lambda_1,\dots,\lambda_n)$  in $P_{\delta_n}$ such that $\lambda_1\geq \cdots \geq \lambda_n$.
Put differently, $\mathrm{Par}_{\leq \delta_n}$ consists of all partitions of size $\binom{n-1}{2}$ and length at most $n$ that are dominated by $\delta_n$ \cite{Rad52}.
Note that all elements in $\mathrm{Par}_{\leq \delta_n}$ do indeed belong to $\widehat{\p C}_{n}$.

\begin{lemma}\label{lem:unique_rep_dominated}
Given $\lambda=(\lambda_1,\dots,\lambda_n)\in \mathrm{Par}_{\leq \delta_n}$, no element in the set $\{\sort\circ \shift^j(\lambda)\suchthat 1\leq j\leq n-1\}$ belongs to $\mathrm{Par}_{\leq \delta_n}$.
\end{lemma}
\begin{proof}
We employ the lattice paths $L_{\lambda}$ and $L_{\lambda}^{\infty}$ defined earlier.
Label the horizontal steps in $L_{\lambda}$ with integers $0$ through $n-1$ going right to left.
Consider the fragment $L'$ of $L_{\lambda}^{\infty}$ of length $2n$  that starts with the horizontal step labeled $j$ and proceeds northwest.
As discussed before, the partition determined by $L'$ (in the bottom left corner of the $n\times n$ box this path naturally lives in) is $\sort\circ \shift^j(\lambda)$.
If we let $i$ denote the number of vertical steps in $L$ preceding the horizontal step labeled $j$, then we have
\begin{align}\label{eqn:size_change_upon_shifting}
	|\sort\circ \shift^j(\lambda)|=|\lambda|+(j-i)n
\end{align}

\medskip

Suppose there exists $j\neq 0$ such that $|\sort\circ \shift^j(\lambda)|\in \mathrm{Par}_{\leq \delta_n}$.
From \eqref{eqn:size_change_upon_shifting} it follows that $j=i$.
Thus, the horizontal step labeled $j$ must touch the diagonal $x+y=n$.
Let $\nu=(\nu_1,\dots,\nu_{n-j})$ be the partition determined by the subpath of $L$ restricted to the $(n-j)\times (n-j)$ box in the top left.
Let $\mu=(\mu_1,\dots,\mu_j)$ be the partition determined by the subpath of $L$ restricted to the $j\times j$ box in the bottom right.
Now, observe that
\begin{align}
\lambda&=(n-j+\mu_1,n-j+\mu_2,\dots,n-j+\mu_{j},\nu_1,\dots,\nu_{n-j})\\
\sort\circ \shift^j(\lambda)&=(j+\nu_1,j+\nu_2,\dots,j+\nu_{n-j},\mu_1,\dots,\mu_j).
\end{align}

\medskip

Since $\lambda$ is dominated by $\delta_n$, we know that
\begin{align}
\sum_{k=1}^{j}(n-j+\mu_k) \leq \sum_{k=1}^{j} (n-j+k-2)=j(n-j-1)+\binom{j}{2}.
\end{align}
Since the left-hand side is $|\lambda|-|\nu|=\binom{n-1}{2}-|\nu|$, we may rewrite the above inequality as
\begin{align}\label{eqn:initial_ineq}
|\nu| \geq \binom{n-1}{2}-j(n-j-1)-\binom{j}{2}=\binom{n-j-1}{2}.
\end{align}
On the other hand, since our assumption is that $\sort\circ \shift^j(\lambda)$ is also dominated by $\delta_n$, we obtain
\begin{align}
\sum_{k=1}^{n-j}(j+\nu_k)\leq \sum_{k=1}^{n-j} (j-2+k)=(n-j)(j-1)+\binom{n-j}{2}.
\end{align}
This in turn may be rewritten as
\begin{align}
\label{eqn:final_ineq}
	|\nu|\leq \binom{n-j}{2}-n+j=\binom{n-j-1}{2}-1,
\end{align}
which is in contradiction with the inequality in \eqref{eqn:initial_ineq}.
\end{proof}

To illustrate, consider $n=10$ and  $\lambda=(7,5,5,5,4,4,2,2,2,0)\vdash 36$ as shown in Figure~\ref{fig:fig_1}.
The path $L_{\lambda}$  is shown in the $10\times 10$ box shaded light gray.
The red line shows the diagonal $y+x=10$ and a horizontal step touching this diagonal gives a value $j$ such that $\sort\circ \shift^j(\lambda)\vdash 36$. In the figure, we have chosen $j=6$.
The path $L'$ is shown in the $10\times 10$ box with the darker shade of gray, and we see that $\sort\circ \shift^6(\lambda)=(8,8,8,6,3,1,1,1,0,0)$.
Note that $\nu=(2,2,2,0)$ is determined by the path in the intersection of the two shaded regions.
Also $\mu=(3,1,1,1,0,0)$ is the partition determined by subpath of $L_{\lambda}$ by the $6\times 6$ box in the bottom right.
The reader may verify that $(8,8,8,6,3,1,1,1,0,0)\notin\mathrm{Par}_{\leq \delta_{10}}$.

\medskip

\begin{figure}
\begin{tikzpicture}[scale=.2]
    \draw[gray,very thin] (-8,-1) grid (14,18);
    \draw[line width=0.15mm, black, <->] (-8,1)--(14,1);
    \draw[line width=0.15mm, black, <->] (1,-1)--(1,18);
    \draw[line width=0.25, red] (13,-1)--(-6,18);
    \draw[fill=gray,opacity=0.2] (1, 1) rectangle (11,11) {};
    \draw[fill=gray,opacity=0.4] (-5, 7) rectangle (5,17) {};

    \draw[fill=blue,opacity=0.25] (1, 1) rectangle (5,7) {};
    \draw[fill=green,opacity=0.25] (5, 1) rectangle (6,5) {};
    \draw[fill=green,opacity=0.25] (6, 1) rectangle (7,2) {};
     \draw[fill=green,opacity=0.25] (7, 1) rectangle (8,2) {};
     \draw[fill=blue,opacity=0.3] (-5, 7) rectangle (1,11) {};
    \draw[fill=green,opacity=0.25] (-5, 11) rectangle (-4,15) {};
    \draw[fill=green,opacity=0.25] (-4, 11) rectangle (-2,12) {};
    \draw[fill=orange,opacity=0.25] (1, 7) rectangle (3,10) {};

    \draw[blue, line width=0.4mm] (11,1)--(11,0)--(13,0)--(13,-1);
    \draw[blue, line width=0.4mm] (11,1)--(8,1);
    \draw[blue, line width=0.4mm] (8,1)--(8,2);
    \draw[blue, line width=0.4mm] (8,2)--(6,2);
    \draw[blue, line width=0.4mm] (6,2)--(6,5);
    \draw[blue, line width=0.4mm] (6,5)--(5,5);
    \draw[blue, line width=0.4mm] (5,5)--(5,7);
    \draw[blue, line width=0.4mm] (5,7)--(3,7);
    \draw[blue, line width=0.4mm] (3,7)--(3,10);
    \draw[blue, line width=0.4mm] (3,10)--(1,10)--(1,11);

    \draw[blue, line width=0.4mm] (1,11)--(-2,11)--(-2,12)--(-4,12)--(-4,15)--(-5,15)--(-5,17)--(-7,17)--(-7,18);
    \draw[blue, line width=0.4mm] (8,1)--(8,2);
    \draw[blue, line width=0.4mm] (8,2)--(6,2);
    \draw[blue, line width=0.4mm] (6,2)--(6,5);
    \draw[blue, line width=0.4mm] (6,5)--(5,5);
    \draw[blue, line width=0.4mm] (5,5)--(5,7);
    \draw[blue, line width=0.4mm] (5,7)--(3,7);
    \draw[blue, line width=0.4mm] (3,7)--(3,10);
    \draw[blue, line width=0.4mm] (3,10)--(1,10)--(1,11);

  \end{tikzpicture}
  \caption{Path corresponding to $\lambda=(7,5,5,5,4,4,2,2,2,0)$.}
  \label{fig:fig_1}
\end{figure}
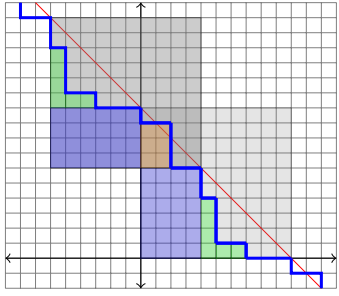

For $\lambda\in \mathrm{Par}_{\leq\delta_n}$, let $\p O_{\lambda}$ denote the $S_n$-orbit of $\lambda$.
Then we know that
\begin{align}\label{eqn:lattice_points_via_orbits}
	\mathrm{Lat}(P_{\delta_n})=\coprod_{\lambda\in \mathrm{Par}_{\leq \delta_n}} \p O_{\lambda}.
\end{align}
By Lemma~\ref{lem:unique_rep_dominated}, we know that each element in $\p O_{\lambda}$ indexes a unique equivalence class in $\widehat{\p C}_{n}$.
Since there are $n^{n-2}$ equivalence classes and this equals the cardinality of the left-hand side in \eqref{eqn:lattice_points_via_orbits}, we infer that the elements of $\p O_{\lambda}$ form a complete set of representatives as $\lambda$ runs over $\mathrm{Par}_{\leq \delta_n}$.
The preceding discussion in conjunction with Remark~\ref{rem:restriction} yields the following result.

\begin{theorem}\label{thm:frob_permutahedron}
	The representation $\gamma_n$  obtained by the $S_n$ action on $\mathrm{Lat}(P_{\delta_n})$ is isomorphic to the representation $\widehat{\tau}_n$ obtained by the $S_n$ action on $\widehat{\p C}_{n}$.
	Furthermore, the restriction of $\gamma_n$ to $S_{n-1}$ is $\mathrm{Park}_{n-1}$.
	The explicit $h$-expansion of $\mathrm{Frob}(\gamma_n)$ may be obtained as follows:
	Suppose $\mathrm{mult}(\lambda)$  denotes the partition recording the multiplicities of each part in $\lambda$ for $\lambda\in \mathrm{Par}_{\leq \delta_n}$ (recall we are allowing $0$ to be a part as well). Then
	\[
	\mathrm{Frob}(\gamma_n)=\mathrm{Frob}(\widehat{\tau}_n)=\sum_{\lambda\in \mathrm{Par}_{\leq\delta_n}}h_{\mathrm{mult}(w)}.
	\]
\end{theorem}

Taking \cite[Theorem 3.1]{KT20} into account, we get the following result.
\begin{corollary}\label{cor:lattice_points_fixed}
Let $\pi\in S_n$ have cycle type $\lambda=(\lambda_1,\dots,\lambda_{\ell})$ where $\lambda_{\ell}>0$.
Set $d\coloneqq \GCD(\lambda_1,\dots,\lambda_{\ell})$. Then the number of lattice  points in $\mathrm{Lat}(P_{\delta_n})$ fixed by the action of $\pi$ is given by $f(d)n^{\ell-2}$ where
\begin{align*}
f(d)=\left\lbrace\begin{array}{ll}
1 & d=1,\\
2 & d=2 \text{ and } n=2 \: (\md 4),\\
0 & \text{otherwise.}
\end{array}\right.
\end{align*}
\end{corollary}
Note that Corollary~\ref{cor:connection_with_equivariant_volumes} now follows in view of \cite[Theorem 1.2]{ASV18}.
\noindent Is there a combinatorial proof of Corollary~\ref{cor:lattice_points_fixed} which eschews character-theoretic computations?

\medskip

We conclude with a couple of remarks.
\begin{remark}\emph{
In private communication with the authors, S.~Backman informed them that lattice points in the trimmed permutahedron $P_{\delta_n}$ may be interpreted as \emph{break divisors} on the complete graph $K_n$, and the latter are in bijection with the set of spanning trees of $K_n$.
Furthermore, in the divisor group of $K_n$, break divisors are in the same equivalence class as  \emph{$q$-reduced divisors}, which turn out be usual parking functions.
We refer the reader to \cite{BN07, ABKS14,Bac17} for more on these beautiful connections.
By appealing to these (non-trivial) results, one could have bypassed our elementary Lemma~\ref{lem:unique_rep_dominated} to arrive at Theorem~\ref{thm:frob_permutahedron}.}
\end{remark}

\begin{remark}\emph{
 One could ask for a  generalization of  Theorem~\ref{thm:frob_permutahedron} when $m>1$. The following example shows that a na\" ive generalization may not work.
Consider $m=2$ and $n=3$. Thus $c_{m,n}=1 \: (\md 3)$.
The orbits of $\p C_{m,n}$ are indexed by the partitions
\begin{align*}
\begin{array} {ccc}
100000 & 111100  & 211000 \\
211111 &  222211 & 221110\\
222220 &  220000 & 222100
\end{array}
\end{align*}
where the partitions in each column are obtained by applying $\sort\circ\shift^j$ to the partition in the top row. Any three partitions, one from each column, index $S_6$-orbits for the action on $\widehat{\p C}_{2,3}$.
It is clear that in this instance that there is no way to pick three such partitions, all of the same size.}
\end{remark}

\section*{Acknowledgements}
This material is based upon work supported by the Swedish Research
Council under
grant no. 2016-06596 while the authors were in residence at Institut
Mittag-Leffler in Djursholm, Sweden during Spring 2020.
We would like to thank the institute for its hospitality during our stay.
The third author benefitted immensely from discussions with Darij Grinberg, Philippe Nadeau and Marino Romero.
Thanks also to Brendon Rhoades, Spencer Backman  and Steven Rayan for very helpful email correspondence(s).

\bibliographystyle{alpha}
\bibliography{Biblio_PS}

\begin{thebibliography}{ABKS14}

\bibitem[ABKS14]{ABKS14}
Yang An, Matthew Baker, Greg Kuperberg, and Farbod Shokrieh.
\newblock Canonical representatives for divisor classes on tropical curves and
  the matrix-tree theorem.
\newblock {\em Forum Math. Sigma}, 2:e24, 25, 2014.

\bibitem[ASVM]{ASV18}
F.~Ardila, A.~Schindler, and A~Vindas-Mel\'endez.
\newblock The equivariant volumes of the permutahedron.

\bibitem[Bac17]{Bac17}
Spencer Backman.
\newblock Riemann-{R}och theory for graph orientations.
\newblock {\em Adv. Math.}, 309:655--691, 2017.

\bibitem[BN07]{BN07}
Matthew Baker and Serguei Norine.
\newblock Riemann-{R}och and {A}bel-{J}acobi theory on a finite graph.
\newblock {\em Adv. Math.}, 215(2):766--788, 2007.

\bibitem[BR14]{BR14}
Andrew Berget and Brendon Rhoades.
\newblock Extending the parking space.
\newblock {\em J. Combin. Theory Ser. A}, 123:43--56, 2014.

\bibitem[Che19]{Ch20}
Shane Chern.
\newblock An extension of a formula of {J}ovovic.
\newblock {\em Integers}, 19:Paper No. A47, 7, 2019.

\bibitem[Hai94]{Hai94}
Mark~D. Haiman.
\newblock Conjectures on the quotient ring by diagonal invariants.
\newblock {\em J. Algebraic Combin.}, 3(1):17--76, 1994.

\bibitem[KT]{KT20}
M.~Konvalinka and V.~Tewari.
\newblock Some natural extensions of the parking space.

\bibitem[Lot97]{Lo97}
M.~Lothaire.
\newblock {\em Combinatorics on words}.
\newblock Cambridge Mathematical Library. Cambridge University Press,
  Cambridge, 1997.
\newblock With a foreword by Roger Lyndon and a preface by Dominique Perrin,
  Corrected reprint of the 1983 original, with a new preface by Perrin.

\bibitem[Pos09]{Pos09}
A.~Postnikov.
\newblock Permutohedra, associahedra, and beyond.
\newblock {\em Int. Math. Res. Not. IMRN}, (6):1026--1106, 2009.

\bibitem[PS04]{Pos04}
A.~Postnikov and B.~Shapiro.
\newblock Trees, parking functions, syzygies, and deformations of monomial
  ideals.
\newblock {\em Trans. Amer. Math. Soc.}, 356(8):3109--3142, 2004.

\bibitem[Rad52]{Rad52}
R.~Rado.
\newblock An inequality.
\newblock {\em J. London Math. Soc.}, 27:1--6, 1952.

\bibitem[Ray18]{Ray18}
Steven Rayan.
\newblock Aspects of the topology and combinatorics of {H}iggs bundle moduli
  spaces.
\newblock {\em SIGMA Symmetry Integrability Geom. Methods Appl.}, 14:Paper No.
  129, 18, 2018.

\bibitem[Sta91]{St91}
Richard~P. Stanley.
\newblock A zonotope associated with graphical degree sequences.
\newblock In {\em Applied geometry and discrete mathematics}, volume~4 of {\em
  DIMACS Ser. Discrete Math. Theoret. Comput. Sci.}, pages 555--570. Amer.
  Math. Soc., Providence, RI, 1991.

\bibitem[Sta99]{St99}
R.~P. Stanley.
\newblock {\em Enumerative combinatorics. {V}ol. 2}, volume~62 of {\em
  Cambridge Studies in Advanced Mathematics}.
\newblock Cambridge University Press, Cambridge, 1999.
\newblock With a foreword by Gian-Carlo Rota and appendix 1 by Sergey Fomin.

\end{thebibliography}

\end{document}